\documentclass[12pt]{amsart}
\usepackage{tgtermes}
\usepackage[linkcolor=blue, citecolor = blue, linktocpage = true]{hyperref}
\hypersetup{colorlinks=true}

\usepackage{amsthm,amsfonts,amsmath,amssymb,amscd} 

\usepackage{verbatim}
\usepackage{float}
\usepackage{wrapfig}
\usepackage{mathtools}
\usepackage[color=orange!20, linecolor=orange]{todonotes}

\usepackage{enumitem}
\usepackage{caption}
\usepackage{subcaption}
\usepackage{lscape}
\usepackage{rotating}

\newtheorem{theorem}{Theorem}[section]
\newtheorem{proposition}[theorem]{Proposition}
\newtheorem{lemma}[theorem]{Lemma}
\newtheorem{corollary}[theorem]{Corollary}
\newtheorem{conjecture}[theorem]{Conjecture}

\theoremstyle{definition}
\newtheorem{definition}[theorem]{Definition}
\newtheorem{example}[theorem]{Example}

\theoremstyle{remark}
\newtheorem{remark}[theorem]{Remark}

\newcommand{\sgn}{\mathrm{sgn}}

\DeclarePairedDelimiter\floor{\lfloor}{\rfloor}

\title[Some unimodal sequences of Kronecker coefficients]
{Some unimodal sequences of Kronecker coefficients} 

\author[Alimzhan Amanov \and Damir Yeliussizov]{Alimzhan Amanov \and Damir Yeliussizov}

\address{KBTU, Almaty, Kazakhstan}
\email{\href{mailto:alimzhan.amanov@gmail.com}{alimzhan.amanov@gmail.com}, \href{mailto:yeldamir@gmail.com}{yeldamir@gmail.com}}

\begin{document}
\begin{abstract}
	We conjecture unimodality for some sequences of generalized Kronecker coefficients and prove it for partitions with at most two columns. 
	The proof is based on a hard Lefschetz property for corresponding highest weight spaces. 
	We also study more general Lefschetz properties, show implications to a higher-dimensional analogue of the Alon--Tarsi conjecture on Latin squares and give related positivity results. 
\end{abstract}
\maketitle
\section{Introduction}

\subsection{Kronecker coefficients and unimodality}
\textit{The (generalized) Kronecker coefficient} denoted by 
$g(\lambda^{(1)},\ldots,\lambda^{(d)})$ is the multiplicity of $[(m)]$ in the tensor product $[\lambda^{(1)}]\otimes\ldots\otimes[\lambda^{(d)}]$, where 
$\lambda^{(i)}$ are partitions of $m$ and $[\mu]$ denotes the irreducible representation of the symmetric group $S_m$ indexed by partition $\mu$ of $m$. 

\vspace{0.5em}

A sequence $\{ x_n\}_{n = a, \ldots, b}$ 
is \textit{unimodal} if $x_a \le \ldots\le x_m \ge \ldots\ge x_b$ for some $m \in [a,b]$, and it is {\it symmetric} if $x_i = x_{a + b - i}$ for all $i \in [a,b]$.

\vspace{0.5em}

Let us denote $g_d(n,k) := g(n\times k,\ldots,n\times k)$ (repeated $d$ times) for specific {rectangular} Kronecker coefficients, where $n \times k := (k,\ldots,k)$ ($n$ times).
We studied these coefficients in \cite{ay22}; in particular, for fixed $k$ and odd $d \ge 3$ the sequence $\{g_{d}(n, k)\}_{n=0,\ldots,k^{d-1}}$ is symmetric (note that $g_d(n,k) = 0$ for $n > k^{d-1}$), and we made the following conjecture on its unimodality. 
\begin{conjecture}[\cite{ay22}]\label{conj:ay}
	Let $d \ge 3$ be odd and $k$ be even. Then the sequence 
	$\{g_{d}(n, k)\}_{n=0,\ldots,k^{d-1}}$ 
	is unimodal.
\end{conjecture}
For example, we have the following sequences: 
$$\{g_3(n, 4) \}_{n = 0, \ldots, 16} = 1, 1, 1, 2, 5, 6, 13, 14, 18, 14, 13, 6, 5, 2, 1, 1, 1,$$
$$\{g_5(n, 2) \}_{n = 0, \ldots, 16} = 1, 1, 5, 11, 35, 52, 112, 130, 166, 130, 112, 52, 35, 11, 5, 1, 1.$$

\vspace{0.5em}

In this paper, we refine this conjecture and prove it for $k = 2$.

\vspace{0.5em}

For a partition $\lambda = (\lambda_1, \ldots)$ and $k \ge \lambda_1$ define the operation 
$$\rho_k : \lambda \mapsto (k, \lambda) = (k, \lambda_1, \ldots),$$ 
which adds the first part $k$ to $\lambda$ (similarly, the inverse operation $\rho_k^{-1}$ removes the first part $k$).
We also apply this operation on tuples of partitions componentwise, i.e. $\rho_k (\lambda^{(1)}, \lambda^{(2)}, \ldots ) = (\rho_k \lambda^{(1)}, \rho_k \lambda^{(2)}, \ldots )$. 

\vspace{0.5em}

We propose the following more general conjecture on unimodality of Kronecker coefficients. 
\begin{conjecture}\label{conj:main}
	Let $d \ge 3$ be odd, $k$ be even and $\pmb{\lambda} = (\lambda^{(1)},\ldots,\lambda^{(d)})$ be a $d$-tuple of partitions of $m \le k^{d}/2$ each with parts at most $k$. Let $a = \max_i (\lambda^{(i)})'_1$ and $b = \min_i (\lambda^{(i)})'_k$. 
	Then the sequence 
		\begin{align}\label{seq1}
		\{ g(\rho_k^n\, \pmb{\lambda} )\}_{n=-b,\ldots,k^{d-1}-a} 
		\end{align}
		is unimodal. 
\end{conjecture}

Note that $g(\rho_k^{n}\, \pmb{\lambda} ) = 0$ for $n \not\in [-b, k^{d-1} - a]$ by trivial reasons.
	Moreover, we show (see Lemma~\ref{lemma:hwvhodge}) that the sequence \eqref{seq1} is symmetric if $m$ is divisible by $k/2$ and 
	each $\lambda = \lambda^{(i)}$ satisfies the complementarity condition 
    $\rho_k^{k^{d-1} - 2m/k}\, \lambda = k^{d-1} \times k - \lambda = (k - \lambda_{k^{d-1}}, \ldots, k - \lambda_1)$, 
    we call such partitions (and their tuples) 
	{\it $k$-complementary}. 
\begin{remark}
This statement essentially splits the set of $d$-tuples of partitions with parts at most $k$ into disjoint sequences, so that the corresponding Kronecker coefficients form unimodal and sometimes symmetric sequences. 
\end{remark}

\begin{remark}
We verified this conjecture computationally on many examples, including all sequences with $d = 3, k = 4$ and $m \le 8$.
We give some examples of such unimodal sequences in Example~\ref{ex42}. Computations show that similar unimodality of Kronecker coefficients frequently hold in the case of odd $k$ as well, although there is an example when it does {\it not} hold: $\{g_3(n, 3)\}_{n = 0, 1, 2, \ldots} = \{1,0,1,\ldots\}$. 
This is the only case ruining unimodality for odd $k$ that we are aware of (from our computations). 
\end{remark}

\begin{remark}
The unimodality can be compared as a specific `vertical' counterpart of the well-known semigroup property for Kronecker coefficients \cite{chm}.
\end{remark}

We shall prove Conjecture~\ref{conj:main} (and hence Conjecture~\ref{conj:ay}) for $k = 2$.
\begin{theorem}
\label{th:main}
Conjecture~\ref{conj:main} holds for $k = 2$.
\end{theorem}

Note that all two-column partitions are $2$-complementary.
In particular, taking $\lambda^{(i)} = \varnothing$ we have $\rho_2^n\, \varnothing = n \times 2 = (2^n)$ and the theorem says that the sequence $\{ g_d(n, 2) \}_{n = 0, \ldots, 2^{d - 1}}$ is unimodal (and symmetric). 

\vspace{0.5em}

We prove Theorem~\ref{th:main} 
by showing the {\it hard Lefschetz property} on highest weight spaces corresponding to Kronecker coefficients.

\subsection{Lefschetz properties (LP)} 
Let $V = (\mathbb{C}^{k})^{\otimes d}$ which is a $k^d$-dimensional space of tensors. 
It is known (see \cite{imw17, bi13}) that for a $d$-tuple $\pmb{\lambda}$ of partitions of $m$ with parts at most $k$ we have
\begin{align*}
	g(\pmb{\lambda}) = \dim \mathrm{HWV}_{\pmb{\lambda'}} \bigwedge^{m} V
\end{align*}
is the dimension of the highest weight vector space of weight $\pmb{\lambda'}$
w.r.t. the induced action of the group $\mathrm{GL}(k)^{\times d}$. 

Define the element $\omega = \omega_{d,k} \in \bigwedge^k V$ which we call the {\it Cayley form}\footnote{As we show in \cite{ay22}, $\omega$ can be viewed as dual to {\it Cayley's first hyperdeterminant} \cite{cay, cay2}} as follows: 
$$
\omega := \sum_{\pi_1,\ldots,\pi_d \in S_k} \sgn(\pi_1\cdots\pi_d) \bigwedge_{i=1}^k e_{\pi_1(i)}\otimes \ldots \otimes e_{\pi_d(i)},
$$
which is a unique highest weight vector of weight $(1 \times k)^d$ (moreover, it is also $\mathrm{SL}(k)^{\times d}$-invariant) \cite{ay22}.

It is natural to approach Conjecture~\ref{conj:main} by establishing a corresponding Lefschetz property. 
Consider the following Lefschetz properties 
depending on $\pmb{\lambda}$ and $k$. (The notation here as in Conjecture~\ref{conj:main}.)
\begin{itemize}
	\item[($\mathrm{LP}_{\pmb{\lambda},k}$)] {\it For some $n_0 \in [-b, k^{d-1}-a]$ the 
		map
		$$
			L: 
			\mathrm{HWV}_{(\rho_k^n\, \pmb{\lambda})'} \bigwedge V \longrightarrow  
			\mathrm{HWV}_{(\rho_k^{n+1}\, \pmb{\lambda})'} \bigwedge V, \qquad 
			L : v \longmapsto \omega \wedge v,
		$$
		is injective for each $n \in [-b, n_0)$ and is surjective for each $n \in [n_0, k^{d-1}-a)$.
		}
\end{itemize} 

\vspace{0.5em}

\begin{itemize}
	\item[($\mathrm{HLP}_{\pmb{\lambda},k}$)]{\it For each $n \in [-b, (k^{d-1} - a + b)/2]$ 
	the map
	$$
		L^{k^{d-1}-a-2n}: \mathrm{HWV}_{(\rho_k^n\pmb{\lambda})'}\bigwedge V \longrightarrow  \mathrm{HWV}_{(\rho_k^{
            k^{d-1} - 2m/k - 2n
            } \pmb{\lambda}
        )'} \bigwedge V, \qquad
		L : v \longmapsto \omega \wedge v,
	$$
	is an isomorphism.
	}
\end{itemize}
\vspace{0.5em}
Let us note that the hard Lefschetz property $\mathrm{HLP}_{\pmb{\lambda}, k}$ only makes sense for $k$-complementary $\pmb{\lambda}$. 
We are interested for which $d, k$ and $\pmb{\lambda}$ the properties $\mathrm{LP}_{\pmb{\lambda},k}$ and $\mathrm{HLP}_{\pmb{\lambda},k}$ hold. It is clear that $\mathrm{HLP}_{\pmb{\lambda}, k}$ implies $\mathrm{LP}_{\pmb{\lambda}, k}$.  

In \textsection\ref{sec:sl2} we prove that the properties $\mathrm{LP}_{\pmb{\lambda},2}$ and $\mathrm{HLP}_{\pmb{\lambda},2}$ hold for all odd $d \ge 3$ and corresponding $\pmb{\lambda}$ (which are all $2$-complementary).
We conjecture that the properties $\mathrm{LP}_{\pmb{\lambda},k}$ and $\mathrm{HLP}_{\pmb{\lambda},k}$ hold for all odd $d \ge 3$ and even $k$.
We show that these properties have two important implications.

\subsubsection{LP implies  Kronecker unimodality}
\begin{proposition}\label{proposition:main-uni}
	If $\mathrm{LP}_{\pmb{\lambda},k}$ holds then Conjecture~\ref{conj:main} holds for $\pmb{\lambda}$. 
\end{proposition}

\subsubsection{HLP implies $d$-dimensional Alon--Tarsi conjecture}
Let us now explain the second implication of the Lefschetz property. 
There are {\it $d$-dimensional Alon--Tarsi numbers} $\mathrm{AT}_d(k) \in \mathbb{Z}$ introduced by B\"urgisser and Ikenmeyer \cite{bi} (for $d = 3$) and studied in  \cite{ay22} (for general $d$), which are equal to the difference of the number of `positive' and `negative' {\it Latin hypercubes}, see  \textsection\ref{subsection:magicsets} for definitions. B\"urgisser and Ikenmeyer \cite{bi} showed that $\mathrm{AT}_3(k)$ is an evaluation of certain fundamental invariant at unit tensor and they raised the problem whether $\mathrm{AT}_3(k) \neq 0$ holds for even $k$ (having positive computational evidence for $k = 2, 4$); see also \cite{lzx} for related work. The celebrated Alon--Tarsi conjecture on Latin squares \cite{at} states that $\mathrm{AT}_2(k) \neq 0$ for even $k$ (see also \cite{hr,dri1,glynn}).
The statement $\mathrm{AT}_d(k) \neq 0$ for even $k$ can be viewed as a $d$-dimensional Alon--Tarsi conjecture. We show that the above Lefschetz property implies this conjecture. 

\begin{proposition}
\label{proposition:main-at}
	Let $d \ge 3$ be odd and $k$ be even. If $\mathrm{HLP}_{\varnothing,k}$ 
	holds then $\mathrm{AT}_d(k) \neq 0$.
\end{proposition}

\subsubsection{More general LP does not hold}\label{introlp}


To indicate significance of the highest weight spaces for the defined Lefschetz properties, let us consider the following more general Lefschetz property $\mathrm{LP}_{d, k}$ depending on $d, k$. 
\begin{itemize}
	\item[($\mathrm{LP}_{d,k}$)] {\it  For each $n \in [0,k^{d}-k]$ the 
		map
		$$
			L: 
			\bigwedge^n V \longrightarrow  
			\bigwedge^{n + k} V, \qquad 
			L : v \longmapsto \omega \wedge v,
		$$
		has full rank. 
		}
\end{itemize}
While we prove that the property $\mathrm{LP}_{d, k}$ holds for $k = 2$, we also show that for $k > 2$ it does {\it not} hold, see \textsection\ref{sec:neg} for details. 


\begin{theorem}
	Let $d \ge 3$ be odd and $k > 2$. Then the property $\mathrm{LP}_{d, k}$ does not hold.
\end{theorem}

\subsubsection{Stable injectivity of the Lefschetz map}
 We prove that the Lefschetz map $L : v \mapsto \omega \wedge v$ satisfies the following stable injectivity.
\begin{theorem}
    Let $d \ge 3$ be odd and $V_n := (\mathbb{C}^n)^{\otimes d}$.
    For $\pmb{\lambda} \subseteq (k^{d-1} \times k)^d$, if the map
    $$L_k: \mathrm{HWV}_{(\pmb{\lambda})'}\bigwedge V_k \longrightarrow \mathrm{HWV}_{(\rho_k \pmb{\lambda})'}\bigwedge V_k, \qquad L_k : v \longmapsto \omega_{d,k} \wedge v, $$ 
    is injective, then for each $\ell > k$, the map
    $$L_\ell: \mathrm{HWV}_{(\pmb{\lambda})'}\bigwedge V_\ell \longrightarrow \mathrm{HWV}_{(\rho_\ell \pmb{\lambda})'}\bigwedge V_\ell, \qquad L_{\ell} : v \longmapsto \omega_{d,\ell} \wedge v, $$
    is also injective.
\end{theorem}
In particular, we have the following corollary.
\begin{corollary}
    Let $d \ge 3$ be odd and $\pmb{\lambda}$ be a $d$-tuple of partitions of $m \le 2^{d-1}$ with parts at most $2$. Then for each $\ell > 2$ the map $L:  \mathrm{HWV}_{(\pmb{\lambda})'}\bigwedge V_{\ell} \to \mathrm{HWV}_{(\rho_\ell \pmb{\lambda})'}\bigwedge V_{\ell}$ is injective.
    In particular, $g(\pmb{\lambda}) \le g(\rho_\ell\, \pmb{\lambda})$.
\end{corollary}

\subsection{Kronecker positivity}
 The (hard) Lefschetz property implies unimodality (and hence positivity) of the symmetric sequence $\{g_d(n,k) \}_{n = 0, \ldots, k^{d - 1}}$ for all odd $d \ge 3$ and even $k$. 
We studied this sequence and conjectured its unimodality in \cite{ay22}, and in particular, we proved the following results on their positivity. 

\begin{theorem}[Cor.~8.7, Thm.~8.13 \cite{ay22}]\label{th:posbefore}
	Let $d \ge 3$ be odd and $k$ be even. 
	Then $g_d(n,k) > 0$ for all $n \le \sqrt{k}/2 - 1$.
	Moreover, if $\mathrm{AT}_d(k) \neq 0$ then $g_d(n,k) > 0$ for all $n \le k^{d-1}$. 
\end{theorem}

In this paper we obtain positivity of rectangular Kronecker coefficients for another range.
\begin{theorem}\label{th:posafter}
	Let $d \ge 3$ be odd and $k$ be even. Then $g_d(n, k) > 0$ for all $n \le 2^{d-1}$. 
\end{theorem}
This theorem gives a better range than the unconditional part of Theorem~\ref{th:posbefore} for large $d$.

\section{Preliminaries}
We denote $[k] := \{ 1, \ldots, k\}$ and $\{e_{i}\}_{i=1}^{k}$ is the standard basis of $\mathbb{C}^{k}$.
\subsection{Partitions} A {\it partition} is a sequence $\lambda = (\lambda_{1}, \ldots, \lambda_{\ell})$ of positive integers $\lambda_1 \ge \cdots \ge \lambda_{\ell}$, where $\ell(\lambda) = \ell$ is its {\it length}. The {\it size} of $\lambda$ is $\lambda_1 + \ldots + \lambda_{\ell}$. Every partition $\lambda$ can be represented as the {\it Young diagram} $\{(i,j)  : i \in [1, \ell], j\in [1,\lambda_{i}]\}$. We denote by $\lambda'$ the {\it conjugate} partition of $\lambda$ whose diagram is transposed. For another partition $\mu = (\mu_1,\ldots,\mu_p)$ we also define partition $\lambda + \mu := (\lambda_1 + \mu_1,\ldots,\lambda_q + \mu_q)$, where $q = \max(\ell, p)$; the sequence $-\lambda = (-\lambda_1,\ldots,-\lambda_\ell)$; and the sequence $\lambda - \mu := \lambda + \mathrm{rev}(-\mu)$, where $\mathrm{rev}(\mu) = (\mu_p,\ldots,\mu_1)$. We extend all operations to $d$-tuples of partitions (sequences) coordinate-wise. 

\subsection{Highest weight vectors} 
Let $V = (\mathbb{C}^{k})^{\otimes d}$ be $k^{d}$-dimensional tensor vector space with the natural (multilinear) action of the group $G=\mathrm{GL}(k)^{\times d}$, which induces the action on the 
anti-symmetric space $\bigwedge^{m} V$. A vector $v \in  \bigwedge^{m} V$ 
is a {\it weight vector} if it is rescaled by the action of diagonal matrices, i.e. 
$$
	(\mathrm{diag}(a^{(1)}_{1},\ldots,a^{(1)}_{k}),\ldots, \mathrm{diag}(a^{(d)}_{1},\ldots,a^{(d)}_{k})) \cdot v = (a^{(1)})^{\lambda^{(1)}}\ldots (a^{(d)})^{\lambda^{(d)}} v,
$$ 
where $x^{\alpha} = x_{1}^{\alpha_{1}}\ldots x_{k}^{\alpha_{k}}$,
so that $(\lambda^{(1)},\ldots,\lambda^{(d)})$ is the {\it weight} of $v$ (here $\lambda^{(i)}$ are not necessarily partitions).

The {\it highest weight vectors} are nonzero weight vectors that are invariant under the group of unipotent matrices $U(k)^{\times d} \subset \mathrm{GL}(k)^{\times d}$.
It is known (see \cite{imw17, bi13}) that for a $d$-tuple $\pmb{\lambda}$ of partitions of $m$ with parts not exceeding $k$ we have
\begin{align}\label{eq:krondim}
	g(\pmb{\lambda}) = \dim \mathrm{HWV}_{\pmb{\lambda'}} \bigwedge^{m} V
\end{align}
is the dimension of the highest weight vector space of weight $\pmb{\lambda'}$
w.r.t. the induced action of $G$. 

We also use an explicit weight space decomposition from \cite{imw17}.
Let the elements of $[k]^{d}$ be ordered lexicographically.  For $X \subseteq [k]^{d}$
denote by $s_{\ell}(X, i)$ the number of elements in $X$ whose $\ell$-th coordinate is $i$, i.e. the number of points of $i$-th {\it slice} in $\ell$-th {\it direction}, and let $s_{\ell}(X) := (s_{\ell}(X,1), \ldots, s_{\ell}(X, k))$ be the vectors of {\it marginals}. 
For $X = \{(x^{(1)}_{i},\ldots,x^{(d)}_{i}) : 1 \le i \le m \}$ ordered lexicograhically we associate the vector
$$
	e_{X} := \bigwedge_{i=1}^{m} e_{x^{(1)}_{i}} \otimes\ldots\otimes e_{x^{(d)}_{i}}.
$$
The vectors $\{e_{X}\}$ over all possible $X \subseteq [k]^{d}$ of cardinality $m$ form a basis of $\bigwedge^{m} V$.
Moreover, $e_{X}$ is a weight vector of weight $(s_1(X), \ldots, s_{d}(X))$ for the action of $G$. 

\subsection{Magic sets and Latin hypercubes}\label{subsection:magicsets}
The \textit{$i$-th slice} of the cube $[k]^d$ in {\it direction} $\ell \in [d]$ is the subset of $[k]^d$ of the elements with $\ell$-th coordinate equal to $i$. We describe definitions from \cite{ay22}.

A \textit{magic set of magnitude} $n$ in $[k]^d$ is a subset $T$ of $[k]^d$ with an equal number of elements in each slice, i.e. for any slice $L$ of $[k]^d$ we have $|L \cap T| = n$. Let $B_{d,k}(n)$ be the set of magic sets of magnitude $n$ in $[k]^d$. 

Magic sets in $[k]^d$ of magnitude $n$ (or total cardinality $nk$) index the basis of $(k \times n)^d$ weight space in $\bigwedge^{nk} V$, i.e. 
	$$
	  \mathrm{WV}_{(k \times n)^d}\bigwedge^{nk} V = \mathrm{span}\{ e_T \mid T \in B_{d,k}(n) \}
	$$
where $T$ is read in lexicographical order.

A {\it partial $d$-dimensional} \textit{Latin hypercube} of type $T$ and {\it magnitude} $n$ is a map $C: T \to [n]$ that is bijective when restricted to any slice of $[k]^d$. 
For a slice $S$ of $[k]^d$, the \textit{$S$-permutation} of $C$ is a permutation of $[n]$ formed by values of the map $C$ on the slice $S$ read in lexicographical order. The set of all partial Latin hypercubes of type $T$ is denoted as $\mathcal{C}_d(T)$; and for $T = [k]^d$ we  denote it $\mathcal{C}_d(k)$.
The \textit{sign} of $C$, denoted by $\sgn(C)$, is the product of signs of $S$-permutations of $C$ over all slices $S$ of the cube $[k]^d$. 
For $T \in B_{d,k}(n)$ the {\it $d$-dimensional Alon-Tarsi number} $\mathrm{AT}(T)$ is defined as the sum of signs of partial Latin hypercubes as follows:
$$
	\mathrm{AT}(T) := \sum_{C \in \mathcal{C}_d(T)} \sgn(C),
$$
and we write $\mathrm{AT}_d(k) := \mathrm{AT}_d([k]^d)$. In other words, $\mathrm{AT}(T)$ is the difference of the number of `positive' and `negative' Latin hypercubes of fixed type $T$.
It is not difficult to show that $\mathrm{AT}(T) = 0$ for odd $k$.

\subsection{Cayley form} Recall the \textit{Cayley form} $\omega =\omega_{d,k} \in \bigwedge^{k} V$ defined as follows: 
$$
\omega = \sum_{\pi_1,\ldots,\pi_d \in S_k} \sgn(\pi_1\cdots\pi_d) \bigwedge_{i=1}^k e_{\pi_1(i)}\otimes \ldots \otimes e_{\pi_d(i)}.
$$
We denote $\omega^{n} := \omega \wedge \cdots \wedge \omega$ ($n$ times). 
In \cite{ay22} we proved the following properties of this vector.
\begin{theorem}[\cite{ay22}]\label{th:omega}
	Let $d \ge 3$ be odd. The following properties hold:
	\begin{itemize}
			\item[(i)] $\omega \in \mathrm{HWV}_{(k\times 1)^d} \bigwedge^{k} V$ is a unique (up to a scale) highest weight vector of such weight;
			\item[(ii)] for odd $k$ we have $\omega^2 = 0$;
			\item[(iii)] for each $n = 1,\ldots,k^{d-1}$ we have the expansion
			$$
				\omega^n = \sum_{T \in B_{d,k}(n)} \pm \mathrm{AT}(T)\, e_T.
			$$
			In particular, $\omega^{k^{d-1}} = \pm \mathrm{AT}_d(k)\, e_{[k]^d}$.
	\end{itemize}
\end{theorem}
In particular, it implies the following corollary on positivity of rectangular Kronecker coefficients $g_d(n,k)$.
\begin{corollary}
For $T \in B_{d,k}(n)$ if $\mathrm{AT}_d(T) \neq 0$ then $g_d(i,k) > 0$ for each $i = 1,\ldots,n$.
\end{corollary}

\section{Highest weight algebra}
We show that for $V = (\mathbb{C}^k)^{\otimes d}$ the highest weight space $\mathrm{HWV} \bigwedge V = \bigoplus_{\pmb{\lambda}} \mathrm{HWV}_{\pmb{\lambda}} \bigwedge V$ can be viewed as a subalgebra\footnote{Highest weight vector by definition is nonzero, but we formally add $0$ to make this set subalgebra.} of $\bigwedge V$, where $\pmb{\lambda}$ runs over $d$-tuples of partitions of $m$. 
This fact is an important ingredient for using the Lefschetz property on the general spaces to obtain Kronecker unimodality in light of  \eqref{eq:krondim}.  

The group $G = \mathrm{GL}(k)^{\times d}$ has the dual representation $V^*$ defined by the action $g \cdot v^* = v^*(g^{-1} x)$ for $g \in G, v \in V$.
If $v$ is the highest weight vector of weight $\lambda$ then $v^*$ is the highest weight vector of weight $-w_0\lambda$ where $w_0$ is the longest element of the corresponding Weyl group. In our case if $v \in \mathrm{HWV}_{\pmb{\lambda}} \bigwedge V$ then $v^* \in \mathrm{HWV}_{\pmb{\lambda^*}} \bigwedge V^*$, where for a partition $\lambda^* = (-\lambda_k,\ldots,-\lambda_1)$ and for a tuple $\pmb{\lambda}^* = ((\lambda^{(1)})^*,\ldots,(\lambda^{(d)})^*)$. We also write $\lambda - \mu := \lambda + \mu^*$. All above extends to tuples of partitions or sequences coordinatewise. 
\begin{lemma}\label{lemma:hwvalg}
	Let $v \in \mathrm{HWV}_{\pmb{\lambda}} \bigwedge^n V$ and $u \in \mathrm{HWV}_{\pmb{\mu}} \bigwedge^m V$. Then
	\begin{enumerate}
		\item (exterior product) $u \wedge v \in \mathrm{HWV}_{\pmb{\lambda}+\pmb{\mu}} \bigwedge^{n+m} V$.
		\item (interior product) if $\pmb{\lambda} - \pmb{\mu} \ge 0$ then $u^*(v) \in \mathrm{HWV}_{\pmb{\lambda} - \pmb{\mu}} \bigwedge^{n-m} V$. 
	\end{enumerate} 
\end{lemma}
\begin{proof}
	Let $E \in U(k)^{\times d}$ be unipotent operator. Then 
	$
		E(v \wedge u) = E(v)\wedge E(u) = v \wedge u
	$,
	i.e. $v\wedge u$ is a highest weight vector of weight $\pmb{\lambda} + \pmb{\mu}$. Similarly, $E(u^*\wedge v) = E(u^*) \wedge E(v) = u^* \wedge v$. The latter relies on the fact that $a^*(b) = a^*(E^{-1}Eb)=(Ea^*)(Eb)$.
\end{proof}

\subsection{Hodge duality}
It is known that the space $\bigwedge V$ has natural {\it Hodge duality}. The \textit{Hodge star operator} is defined as contraction with volume form: 
\begin{align*}
	\star: \bigwedge^{n}V &\to \bigwedge^{k^{d}-n} V,\\
	v &\mapsto v^* \wedge \mathrm{vol}
\end{align*} 
for any $n \in [k^d]$, where $\mathrm{vol} := e_{[k]^d}$ is the volume form, which also a highest weight vector. It is known that $\star \circ \star = (-1)^{n(k^d-n)}\mathrm{id}$ is a scalar multiple of an identity operator. 

\begin{lemma}\label{lemma:hwvhodge}
	Let $\pmb{\lambda}$ be a $d$-tuple of partitions of $m$ with parts at most $k$ and lengths at most $k^{d-1}$. 
	Then we have following symmetry of Kronecker coefficients: 
	$$g(\pmb{\lambda}) = g((k^{d-1}\times k)^d  - \pmb{\lambda}).$$ 
\end{lemma}
\begin{proof}
	By Lemma~\ref{lemma:hwvalg}, the Hodge star operator sends the highest weight vector to the highest weight vector (if non-zero). But since $\star\circ\star = \pm \mathrm{id}$,  the operator $\star$ defines an isomorphism between the highest weight spaces $\star: \mathrm{HWV}_{\pmb{\lambda'}}\bigwedge V \to \mathrm{HWV}_{(k \times k^{d-1})^d - \pmb{\lambda'}}\bigwedge V$.
\end{proof}

\begin{remark}
For $d=3$ this symmetry of Kronecker coefficients was shown in \cite[Cor.~4.4.15]{iken13} (by a different argument). Here we show it for all odd $d$ and provide a direct proof. 
This lemma can also be generalized if we replace $V=(\mathbb{C}^k)^{\otimes d}$ to more general space $\mathbb{C}^{k_1}\otimes\ldots\otimes \mathbb{C}^{k_d}$.
\end{remark}

\section{Proofs of Propositions~\ref{proposition:main-uni} and \ref{proposition:main-at}}
We will now provide the proofs of Proposition~\ref{proposition:main-uni} and Proposition~\ref{proposition:main-at}.

\begin{proof}[Proof of Proposition~\ref{proposition:main-uni}] 
	Without loss of generality, we can shift the sequence and assume $b = 0$ (if $b > 0$ then replace $\pmb{\lambda} \to \rho_k^{-b}\,\pmb{\lambda}$), since the statement $\mathrm{LP}_{\pmb{\lambda},k}$ is common for each element of the sequence $\{ \rho^n_k\, \pmb{\lambda} \}_{n \in [-b,k^{d-1}-a]}$.
	
	Let us denote the spaces 
	$$U_{n} := \mathrm{HWV}_{\pmb{\lambda'}+(k\times n)^d} \bigwedge^{m+nk} V, \qquad n = 0,\ldots,k^{d-1}-a,$$ 
	so that $$\dim U_n = g((\pmb{\lambda'}+(k \times n)^d)') = g(\rho_k^n\, \pmb{\lambda}).$$ 
	By Lemma~\ref{lemma:hwvalg} we have
	\begin{align*}
		L: U_{n} \to U_{n+1},\quad
		v \mapsto \omega\wedge v.
	\end{align*}
	By $\mathrm{LP}_{\pmb{\lambda}, k}$ the sequence of maps
    $$U_{0}\xrightarrow{~~L~~} U_{1}\xrightarrow{~~L~~}\ldots \xrightarrow{~~L~~} U_{k^{d-1}-a}$$
    is injective up to some $n_0 \in [0, k^{d-1}-a]$ and surjective afterwards. This establishes unimodality of the sequence $\{g(\rho_k^n\, \pmb{\lambda}) \}_{n=0,\ldots,k^{d-1}-a}$.
\end{proof}
\begin{remark}\label{rem41}
	Each $d$-tuple $\pmb{\lambda}$ of partitions of $m$ that fit in the rectangle $k^{d-1}\times k$ produces an element of certain unimodal sequence of corresponding Kronecker coefficients, and these sequences do not intersect. Indeed, let $\lambda$ be a partition in  $\pmb{\lambda}$ with the largest $\ell(\lambda) \to \max$. 
	Then we construct the corresponding Kronecker sequence starting from $g(\pmb{\lambda})$ and appending to the left $d$-tuples of partitions resulted by removing the (largest) part $k$ (until possible); and appending to the right of $g(\pmb{\lambda})$ partitions resulting by inserting part $k$ (until possible). 
\end{remark}

\begin{proof}[Proof of Proposition~\ref{proposition:main-at}]
	Using the hard Lefschetz property $\mathrm{HLP}_{\varnothing, k}$ for $n = k^{d-1}/2$ and the result in Theorem~\ref{th:omega}(iii) we obtain 
	$$L^{k^{d-1}}( 1 ) = \omega^{\wedge k^{d-1}} = \pm \mathrm{AT}_d(k)\cdot \mathrm{vol} \neq 0,$$ 
	and hence $\mathrm{AT}_d(k) \neq 0$.
\end{proof} 

\section{Proof of Theorem~\ref{th:main} via $\mathfrak{sl}(2)$ representation}\label{sec:sl2} %
In this section we specialize $k = 2$ so that $V = (\mathbb{C}^{2})^{\otimes d}$ for odd $d \ge 3$. We will first prove that $\bigwedge V$ is an $\mathfrak{sl}(2)$ representation.

Recall that the Lie algebra $\mathfrak{sl}(2)$ is given by its basis $X, Y, H$ called \textit{raising, lowering and counting operators} subject to commutation relations
\begin{align}\label{relations}
[X, Y] = H, \quad [H, Y] = -2Y, \quad [H, X] = 2X.
\end{align}
Thus, we are enough to show that $\mathrm{End}(\bigwedge V)$ contains a Lie subalgebra isomorphic to $\mathfrak{sl}(2)$.
In fact, we will show it 
w.r.t. the carefully chosen operators $X,Y$ that in turn have the highest weight algebra as an invariant subspace. 

\vspace{0.5em}

The key is to present $\omega$ 
in a suitable way. 

\begin{definition}
	Let $\{e_1, e_2\}$ be the standard basis of $\mathbb{C}^2$. Then $V = (\mathbb{C}^2)^{\otimes d} = \langle  e_{\mathbf{i}} : \mathbf{i} \in [2]^d \rangle$ where $e_{\mathbf{i}} = e_{i_1}\otimes\ldots\otimes e_{i_d}$ for $\mathbf{i} = (i_1,\ldots,i_d)$. For $\mathbf{i} = (i_1,\ldots,i_d) \in [2]^d$ denote $|\mathbf{i}| := i_1 + \ldots + i_d - d$. Let us separate the \textit{`positive'} coordinates $I^{+} := \{ \mathbf{i} \in [2]^d : |\mathbf{i}| \text{ is even} \}$ and the {\it `negative'} $I^{-} := [2]^d \setminus I^{+}$. For $\mathbf{i} \in I^{+}$ denote $\bar{\mathbf{i}} := (3-i_1,\ldots,3-i_d) \in I^{-}$. For instance, $(1,\ldots,1) \in I^{+}$ and its negation $(2,\ldots,2) \in I^{-}$. 
	It is clear that this `negation' of indices is a bijection, i.e.  $\overline{I^{+}} = I^{-}, \overline{I^{-}} = I^{+}$.
\end{definition}

\begin{lemma}
 For $k = 2$ the {Cayley} vector $\omega =\omega_{d,2}$ and its dual $\omega^{*}$ can be written as follows:
	$$
		\omega = \sum_{\mathbf{i} \in I^{+}}  e_{\mathbf{i}} \wedge  e_{\,\bar{\mathbf{i}}} ,\qquad \omega^{*} = \sum_{\mathbf{i} \in I^{+}}  e^{*}_{\mathbf{i}} \wedge  e^{*}_{\,\bar{\mathbf{i}}}.
	$$
\end{lemma}
\begin{proof}
	By the definition of $\omega= \omega_{d,2}$ we have
	$$
	\omega = \sum_{\pi = (\pi_1,\ldots,\pi_d) \in (S_2)^d} \sgn(\pi_1)\cdots\sgn(\pi_d)\, e_{\pi(1)} \wedge e_{\pi(2)}.
	$$
	But note that $\pi(1) = (\pi_1(1),\ldots,\pi_d(1)) = (3-\pi(2),\ldots,3-\pi(2)) = \overline{\pi(2)}$ and the sign $\sgn(\pi_1)\cdots\sgn(\pi_d) = (-1)^{|\pi(1)|}$.
	Hence 
	$$
	\omega = \sum_{(1,\mathbf{i}) \in I^{+}} e_{(1,\mathbf{i})} \wedge e_{(2,\,\bar{\mathbf{i}})} - \sum_{(2,\mathbf{i}) \in I^{+}} e_{(1,\bar{\, \mathbf{i}})} \wedge e_{(2,\mathbf{i})} = \sum_{\mathbf{i} \in I^{+}} e_{\mathbf{i}} \wedge e_{\,\bar{\mathbf{i}}}
	$$
	and similarly for $\omega^{*}$.
\end{proof}
	Now using these elements we define the raising operator $X = L : \bigwedge V \to \bigwedge V$ by the left exterior product with $\omega$, the lowering operator $Y: \bigwedge V \to \bigwedge V$ by the left interior product with $\omega^{*}$, i.e.
	\begin{align*}
		X : 
		v \mapsto \omega \wedge v, \qquad
		Y : 
		v \mapsto \omega^{*} \wedge v,
	\end{align*}
	and the {counting operator} $H$ that reduces to multiplication by scalar $(\ell - 2^{d-1})$ on $\bigwedge^{\ell}V$.

\begin{lemma}
	The operators $X,Y,H$ satisfy the commutation relations \eqref{relations}.
\end{lemma}
\begin{proof}
	Since the element $\omega$ is now written in a convenient way, the same proof as can be found in \cite[Ch.~VIII, p.~207]{bourbaki} works here, which we present for completeness. 
	
	Choose the total lexicographic order $<$ on the set of our indices $[2]^d$. Let us select the following basis for the space $\bigwedge V$:
	$$
		 e_{A,B,C} =  e_{a_1}\wedge\ldots\wedge e_{a_p}\wedge  e_{\overline{b_1}}\wedge\ldots\wedge  e_{\overline{b_q}} \wedge  e_{c_1}\wedge  e_{\overline{c_1}}\wedge\ldots\wedge  e_{c_m}\wedge  e_{\overline{c_m}}
	$$
	where $A = \{a_1<\ldots < a_p\}, B = \{b_1 < \ldots < b_q\}, C = \{c_1 < \ldots < c_m\}$ are disjoint subsets of $I^{+}$. 
	Using the fact that $v \wedge  e_{\mathbf{i}} \wedge  e_{\,\bar{\mathbf{i}}} =  e_{\mathbf{i}} \wedge  e_{\,\bar{\mathbf{i}}} \wedge v$ it is easy to verify that
	\begin{align*}
		X   e_{A,B,C} &= \sum_{\mathbf{j} \in I^{+} \setminus A\cup B\cup C}  e_{A,B,C\cup \{\mathbf{j} \}},\\
		Y   e_{A,B,C} &= \sum_{\mathbf{j} \in C}  e_{A,B,C \setminus \{ \mathbf{j} \}}.
	\end{align*}
	This allows us to calculate commutators. Assume $ e_{A,B,C} \in \bigwedge^{\ell} V$, i.e. $p + q + 2m = \ell$; then we have
	\begin{align*}
		[X,Y](e_{A,B,C}) 
		&= 
		\sum_{\mathbf{j}_1 \in C} \, 
		\sum_{\mathbf{j}_2 \in I^{+} \setminus A\cup B\cup (C \setminus \{\mathbf{j}_1\})}
			 e_{A,B, (C \setminus \{ \mathbf{j}_1\}) \cup \{ \mathbf{j}_2\}} \\
		&- 
		\sum_{ \mathbf{j}_2 \in I^{+} \setminus A\cup B\cup C} \,
		\sum_{ \mathbf{j}_1 \in (C \cup \{ \mathbf{j}_2\})}  e_{A,B,(C \cup \{ \mathbf{j}_2\}) \setminus \{ \mathbf{j}_1\}}.
	\end{align*}
	If $\mathbf{j}_1$ and $\mathbf{j}_2$ differ then the coefficient at $\tilde e_{A,B,(C \cup \{ \mathbf{j}_2\}) \setminus \{ \mathbf{j}_1\}}$ is $0$, since the order of adding $\mathbf{j}_2$ and removing $\mathbf{j}_1$ does not matter. If $\mathbf{j}_1 = \mathbf{j}_2 = \mathbf{j}$, the coefficient coming from the first sum is $m$ and from the second sum is $2^{d-1}-p-q-m$, thus 
	$$[X,Y](e_{A,B,C}) = (\ell - 2^{d-1}) e_{A,B,C} = H e_{A,B,C}.$$ The remaining two relations are straightforward.
\end{proof}
\begin{definition}

   Define $P^{n} = \ker(Y)\cap \bigwedge^n V$ called the \textit{primitive class} of degree $n$, where we set $P^{n} = 0$ for $n < 0$. The elements of $P^n$ are called \textit{primitive vectors}.
\end{definition}
The following are standard known properties of $\mathfrak{sl}(2)$ representation. 
\begin{proposition}\label{sl2property}
	We have:
	\begin{enumerate}[label=(\roman*)]
		\item $[X^i, Y](v) = i(n - 2^{d-1} + i - 1) X^{i-1}(v)$ for $v \in \bigwedge^n V$.
		\item For $n \in [0, 2^{d}]$
		$$
			\bigwedge^{n} V = \bigoplus_{i = 0}^{\floor{n/2}} X^i P^{n - 2i},\quad \bigwedge V = \bigoplus_{n \ge 0} \bigoplus_{i = 0}^{2^{d-2}-n}X^{i}(P_n)
		$$
		called {\it Lefschetz decomposition}.
		\item $P^n = 0$ for $n > 2^{d-1}$.
	\end{enumerate}
\end{proposition}
\begin{proof}
	(i) For $i = 1$, the statement is proved earlier. Now by induction on $i$ we get
	\begin{align*}
		[X^i, Y](v) 
			&= (X^i Y - Y X^i)(v)
			\\&= (X^i Y - XYX^{i-1} + XYX^{i-1} - YX^i)(v)
			\\&= X[X^{i-1}, Y]v + [X,Y]X^{i-1}(v)
			\\&= ((i-1)(n - 2^{d-1} + i - 2) + n + 2(i - 1) - 2^{d-1}) X^{i-1}(v)
			\\&= i(n - 2^{d-1} + i - 1) X^{i-1}(v).
	\end{align*}
	
(ii) Since $\bigwedge V$ is a finite-dimensional $\mathfrak{sl}(2)$ representation, it is a direct sum of irreducible $\mathfrak{sl}(2)$ representations. Any irreducible representation $W \subseteq \bigwedge V$ contains a primitive vector $v \in P^n$ of degree $n$, i.e. $Yv = 0$ (if $Yv \neq 0$ then replace $v$ with $Y^{\ell}v$ for large enough $\ell$). Let $\ell$ be minimal integer such that $X^{\ell+1} v = 0$. Then by (i) the space $\langle v, Xv,\ldots,X^{\ell}v \rangle$ is stable under the action of $X$,$Y$ and $H$, hence is irreducible $\mathfrak{sl}(2)$ representation and any irreducible is of that form. Moreover, by (i) with $i \to \ell + 1$ we get $0=[X^{\ell+1},Y](v) = (\ell+1)(n-2^{d-1}+\ell) X^{\ell}(v)$, hence $\ell = 2^{d-1}-n$ and the irreducible representation $W$ is stretched from $\bigwedge^n V$ up to $\bigwedge^{2^d - n}V$. This implies the decomposition.
	
(iii) Let $v \in P^n$ and $i$ be minimal with $X^{i}v = 0$. For $n > 2^{d-2}$, by (i) we have $[X^i, Y](v) = i(2n-2^{d-1}+i-1) X^{i-1}(v)$, which implies $i = 0$, i.e. $v = 0$.
\end{proof}
An immediate implication of the Lefschetz decomposition is the following.
\begin{corollary}
    Let $d \ge 3$ be odd. Then the property $\mathrm{LP}_{d,2}$ holds.
\end{corollary}

\subsection{Highest weight subspace}
Let $\pmb{\lambda}$ be a $d$-tuple of partitions with at most two columns; for simplicity let us also assume that $(\lambda^{(1)})' = (m)$ for $m \le 2^{d-1}$ (see Remark~\ref{rem41}). Note that we have 
$$
	g(\rho_2^{n}\, \pmb{\lambda}) = \dim \mathrm{HWV}_{\pmb{\lambda'} + (2\times n)^d} \bigwedge^{m+2n} V. 
$$
Let us denote the spaces
$$U_{m+2n} := \mathrm{HWV}_{\pmb{\lambda'} + (2\times n)^d}\bigwedge^{m + 2n} V,\qquad U := \bigoplus_{i=0}^{2^{d-1}-m} U_{m+2n}$$
for $n = 0, \ldots, 2^{d-1}-m$. 

Recall that the vector $\omega$ is a unique highest weight vector of weight $(2\times 1)^d$ in the space $\bigwedge^2 V$.

\begin{lemma}
	The space $U$ is an $\mathfrak{sl}(2)$ representation.
\end{lemma}
\begin{proof}
	Since $X$ and $Y$ multiply and contract with the highest weight vector $\omega = \omega_{d,2}$, by Lemma~\ref{lemma:hwvalg} we have: $X(U_i) \subseteq U_{i+1}, Y(U_i) \subseteq U_{i-1}$ and $H(U_i) \subseteq U_i$. Therefore, $U$ is an invariant subspace for the action of $\mathfrak{sl}(2)$, hence is a representation itself.
\end{proof}

The above result implies Theorem~\ref{th:main} by standard properties of $\mathfrak{sl}(2)$. Let us illustrate how. Denote $Q^{m+2n} := P^{m+2n} \cap U$ as primitive class of representation $U$. Then (ii) and (iii) of Proposition~\ref{sl2property} does also hold if $\bigwedge V$ and $P^\ell$ is replaced by $U$ and $Q^\ell$.

\begin{corollary}[Hard Lefschetz property for $k=2$]\label{cor:hlp}
	 The map $X^{\ell}: U_{2^{d-1} - \ell} \to U_{2^{d-1} + \ell}$ is an isomorphism for $\ell = 0,\ldots, m/2$ if $m$ is even and for $\ell = 1, \ldots, \floor{m/2}$ if $m$ is odd. 
\end{corollary}
\begin{proof}
	By the Lefschetz decomposition in Proposition~\ref{sl2property}(ii) we have:
	\begin{align*}
		U_{2^{d-1}+\ell} &= \bigoplus_{i \ge 0} X^i Q^{2^{d-1}+\ell-2i} 
		\\
		&= X^{\ell} \left(
			\bigoplus_{i \ge 0} X^i q^{2^{d-1}-\ell-2i}
		\right) \bigoplus \left(
			\bigoplus_{i = 0}^{\ell-1} X^i Q^{2^{d-1}+\ell-2i}
		\right)
		\\
		&= X^{\ell} U_{2^{d-2}-\ell} \bigoplus 0 = X^{\ell} Q_{2^{d-2}-\ell}
	\end{align*}
	where the terms $X^i Q^{2^{d-1}+\ell-2i}$ vanish: for $i < \ell/2$ since $Q^{2^{d-2}+\ell-i} = 0$ by Proposition~\ref{sl2property}(iii), and for $\ell/2 \le i < \ell$ since for $v \in Q^{2^{d-2}+\ell-2i}$ by Proposition~\ref{sl2property}(ii) $X^i v = 0$. By Lemma~\ref{lemma:hwvhodge} we have $\dim U_{2^{d-2}-\ell} = \dim U_{2^{d-2}+\ell}$ and the statement follows.
\end{proof}

\begin{proof}[Proof of Theorem~\ref{th:main}]
By Corollary~\ref{cor:hlp} the map $U_{m+2i} \xrightarrow{X} U_{m+2(i+1)}$ is injective for $n < (2^{d-1}-m)/2$ and surjective for $n \ge (2^{d-1}-m)/2$. Note that for a partition $\lambda = \lambda^{(i)}$ with at most two columns $k$-complementary condition always holds: $2 \times 2^{d-2} - \lambda' = \lambda' + 2 \times (2^{d-2}-m)$.
Therefore, by Proposition~\ref{proposition:main-uni} the sequence is symmetric and unimodal.
\end{proof}

\begin{remark}
We can similarly define the operators $X,Y,H$ for $k = 4$ using $\omega_{d,4}$. However, the commutation relation $[X,Y] = H$ from \eqref{relations} does {\it not} hold in this case, i.e. these operators do not form an $\mathfrak{sl}(2)$ representation. For instance, if we take $T = \{ (111),(222) \} \subseteq [4]^3$ then 
$$[X,Y](e_T) = 148 e_{\{(111),(222)\}} + 4 e_{\{(112),(221)\}} + 4 e_{\{(121),(212)\}} - 4 e_{\{(122),(211)\}}
$$ 
and so the commutator does not rescale this vector. 
\end{remark}

\begin{remark}
	Another implication of $\mathfrak{sl}(2)$ structure on $\bigwedge V$ with the element $\omega_{d,2}$ is the following unimodality of magic sets in the cube $[2]^d$. Let $b_d(n) := |B_{d,2}(n)|$ be the number of magic sets of magnitude $n$ in $[2]^d$ (see Section \ref{subsection:magicsets}). Then the sequence $\{ b_{d}(n) \}_{n=0,\ldots,2^{d-1}}$ is unimodal (and symmetric) since the space of weight vectors $\bigoplus_{n} \mathrm{WV}_{(2 \times n)^d} \bigwedge^{2n} V$ is preserved by the $\mathfrak{sl}(2)$ triplet $(X,Y,H)$.
\end{remark}

\section{More general LP does not hold}\label{sec:neg}

We now prove that the Lefschetz property $\mathrm{LP}_{d, k}$ (defined in \textsection\ref{introlp}) does not hold for $k > 2$.

\begin{theorem}\label{th:neg}
	Let $d \ge 3$ be odd and $k > 2$. Then the property $\mathrm{LP}_{d, k}$ does not hold.
\end{theorem}
\begin{proof}
    By unimodality of the sequence $\{\dim\bigwedge^n V\}_{n \in [0,k^d]} = \{\binom{k^d}{n}\}_{n\in [0,k^d]}$ the map 
    $$L: \bigwedge^n V \to \bigwedge^{n + k}V$$
    must be injective for $n \le (k^{d}-k)/2$ and surjective afterwards. 
	We will show that the map $L$ has a non-empty kernel for some $n < (k^{d}-k)/2$. 
	
	Consider the following vector 
	$$v := \bigwedge_{\mathbf{i} \in [k]^{d-1}} e_{1\mathbf{i}}$$
	which is a basis vector corresponding to the set of all entries $\mathbf{i}$ of the cube $[k]^d$ that lie in the first (or any) slice of the first direction. Note that $v$ is a $\mathrm{GL}(k)^{\times d}$ highest weight vector, since it nulls under the action of corresponding raising operators (see e.g. for a similar argument \cite{imw17}). Then 
	$$
		L(v) = \omega \wedge v = \sum_{\pi = (\pi_1,\ldots,\pi_d) \in (S_k)^d} \sgn(\pi)\bigwedge_{i=1}^k e_{\pi(i)} \wedge v = 0
	$$
	since for some $j$ we have $\pi_1(j) = 1$ and $e_{\pi(j)} \wedge v = 0$.
\end{proof}
\begin{remark}
    We can construct many other highest weight vectors like $v$ by appending the highest weight vectors to the slices $2,\ldots,k$ in the first direction; all of them will null under $L$ by the same reason.
\end{remark}

\begin{remark}
Let us now make a few remarks on the related property $\mathrm{LP}_{\pmb{\lambda},k}$.
\begin{itemize}
    \item[(i)] The vector $v$ in the proof above has weight $\pmb{\lambda}' = (1 \times k^{d-1}, (k \times k^{d-2})^{d-1})$. Then $g(\pmb{\lambda}) = 1$ and the property $\mathrm{LP}_{\pmb{\lambda},k}$ gives unimodality of the sequence with only element.

    \item[(ii)] Choose $n = \lfloor\frac{d-1}{d}k\rfloor + 1$. 
    Let us take $\pmb{\lambda} = (n^{d-1} \times n)^d$ and let $u = e_{[n]^d} = \bigwedge_{\mathbf{i} \in [n]^d} e_{\mathbf{i}}$. Then by analogous reasoning $u \in \mathrm{HWV}_{(\pmb{\lambda})'}\bigwedge V$.  
    Note that
    $$
        L(v) = \omega \wedge v = \sum_{\pi = (\pi_1,\ldots,\pi_d) \in (S_k)^d} \sgn(\pi)\bigwedge_{i=1}^k e_{\pi(i)} \wedge v = 0,
    $$
    since for each $\pi \in (S_k)^d$ there is $i \in [k]$ with $\pi(i) \in [n]^d$;
    indeed, the $d$ permutations $\pi_1, \ldots, \pi_d$ in $\pi$ have $d(k-n)$ numbers larger than $n$, and since $d(k-n) < k$ there must be $i \in [k]$ for which $\pi(i) = (\pi_1(i), \ldots, \pi_d(i)) \in [n]^d$. 
    Hence, $L(v) = 0$, but the property $\mathrm{LP}_{\pmb{\lambda}, k}$ will still hold if we have the corresponding sequence $\{g(\rho_k^\ell\, \pmb{\lambda})\}_{\ell \ge 0} = \{1, 0, 0, \ldots\}$ (which holds e.g. for $k = 4$ and $n = 3$). 
    
    \item[(iii)] Finally, for odd $k$ the situation is different. For instance, the property $\mathrm{LP}_{\varnothing, k}$ does not hold, since $L(\omega) = \omega^2 = 0$ by Theorem~\ref{th:omega}(ii).
\end{itemize}
\end{remark}

\section{Stable injectivity of Lefschetz map}

\begin{theorem}\label{theorem:injectivity}
    Let $d \ge 3$ be odd and $V_n := (\mathbb{C}^n)^{\otimes d}$. For $\pmb{\lambda} \subseteq (k^{d-1} \times k)^d$, 
    if the map
    $$L_k: \mathrm{HWV}_{(\pmb{\lambda})'}\bigwedge V_k \longrightarrow \mathrm{HWV}_{(\rho_k \pmb{\lambda})'}\bigwedge V_k, \qquad L_k : v \longmapsto \omega_{d,k} \wedge v, $$ 
    is injective, then for each $\ell > k$, the map
    $$L_\ell: \mathrm{HWV}_{(\pmb{\lambda})'}\bigwedge V_\ell \longrightarrow \mathrm{HWV}_{(\rho_\ell \pmb{\lambda})'}\bigwedge V_\ell, \qquad L_{\ell} : v \longmapsto \omega_{d,\ell} \wedge v, $$
    is also injective.
\end{theorem}

Let $\ell > k$. There is a natural inclusion $V_k \subseteq V_{\ell}$ described as follows: if $V_k = \langle e'_{\mathbf{i}} \rangle_{\mathbf{i} \in [k]^d}$ and $V_\ell = \langle e_{\mathbf{i}} \rangle_{\mathbf{i} \in [\ell]^d}$ (here $e'$, $e$ are standard bases) then we map $e'_{\mathbf{i}} \mapsto e_{\mathbf{i}}$ by associating the cube $[k]^d$ with the corresponding subcube of $[\ell]^d$. This also induces inclusion $\bigwedge V_k \subseteq \bigwedge V_{\ell}$. For $v \in V_k$ we will write $v \in V_\ell$ according to described inclusion. In particular, if $v \in \mathrm{HWV} \bigwedge V_k$ then $v \in \mathrm{HWV} \bigwedge V_\ell$, due to induced inclusion of the groups $\mathrm{GL}(k)^{\times d} \to \mathrm{GL}(\ell)^{\times d}$, described by appending ones on the main diagonal to each matrix.

Let $I = (I_1,\ldots,I_d)$ be a $d$-tuple of $k$-sets $I_j =  \{x_{j,1} < \ldots < x_{j,k}\} \in \binom{[\ell]}{k} := \{X : X \subseteq [\ell], |X| = k \}$. Denote by $\omega_I \in \bigwedge V_\ell$ the image of $\omega_{d,k}$ under the map 
$$
e_{i_1} \otimes \ldots \otimes e_{i_d} \mapsto e_{x_{1,i_1}} \otimes \ldots \otimes e_{x_{d,i_d}},
$$
in other words, we associate the cube $[k]^d$ where $\omega_{d,k} \in \bigwedge V_k$ as a subcube of length $k$ in $[\ell]^d$ defined by indices from $I$.

\begin{lemma}\label{lemma:injectivity}
    Let $d \ge 3$ be odd and $v \in \mathrm{WV}_{\pmb{\lambda}} \bigwedge V_k$. Assume that $\omega_{d,k} \wedge v \neq 0$. Then for all $\ell > k$ we also have
    $\omega_{d,\ell} \wedge v \neq 0$. 
\end{lemma}
\begin{proof}
    We can decompose 
    \begin{align}\label{eq:decompose}
		\omega_{d,\ell} \wedge v = \sum_{I = (I_1,\ldots,I_d) \in \binom{[\ell]}{k}^d, I_1 = [k]} \pm ~ \omega_{I} \wedge \omega_{\overline{I}} \wedge v
	\end{align}
	where $\overline{I} := ([\ell] \backslash I_1,\ldots,[\ell] \backslash I_d)$.
	We will show that the vectors $\{\omega_{I}\wedge\omega_{\overline{I}}\wedge v : I \in \binom{[\ell]}{k}^d, I_1 = [k]\}$ are linearly independent. Further, since for $I = [k]^d$ we have $\omega_{I} = \omega_{d,k}$, then $\omega_I \wedge \omega_{\overline{I}} \wedge v \neq 0$ implies $\omega_{d,\ell} \wedge v \neq 0$.

For $u \in \bigwedge V_{\ell}$ denote $\mathrm{supp}(u) := \{ T: T \subseteq [\ell]^d, \langle e_T, u \rangle \neq 0 \}$, i.e. the set of indices of coefficients in the standard basis expansion of the vector $u$.
	
	To show the linear independence, we show a stronger statement: the sets $\{\mathrm{supp}(\omega_{I}\wedge\omega_{\overline{I}}\wedge v) : I \in \binom{[\ell]}{k}^d, I_1 = [k]\}$ are pairwise disjoint. 
    Let $I \neq J$ with $I_a \neq J_a$ for some $a\in [d]$.
	We have $\mathrm{supp}(\omega_I\wedge \omega_{\overline{I}} \wedge v) \cap \mathrm{supp}(\omega_J\wedge \omega_{\overline{J}} \wedge v) = \varnothing$, 
	since their elements have distinct marginals in the $a$-th direction: the elements of the first family have common marginals equal to $\lambda^{(a)} + \delta_{I_a}$ and the elements of the second family -- equal to $\lambda^{(a)} + \delta_{J_a}$, where $\delta_I$ denotes an $\ell$-vector with ones in positions from $I$ and $0$ elsewhere.

	Since the supports are mutually disjoint, the vectors in the sum \eqref{eq:decompose} are linearly independent. Hence, $\omega_{d,k} \wedge  v \neq 0$ implies $\omega_{[k]^d} \wedge \omega_{[k+1,\ell]^d} \wedge v \neq 0$ which was shown to imply $\omega_{d,\ell} \wedge v \neq 0$.
\end{proof}

\begin{proof}[Proof of Theorem~\ref{theorem:injectivity}]
	Note that if $v \in \mathrm{HWV}_{(\pmb{\lambda})'}\bigwedge V_{\ell}$, then we have $v \in \mathrm{HWV}_{(\pmb{\lambda})'}\bigwedge V_k$, since $\pmb{\lambda} \subseteq (k^{d-1} \times k)^d$. From the injectivity of the map $L_k$, for $v \in \mathrm{HWV}_{(\pmb{\lambda})'}\bigwedge V_k$ we have $L_k(v) = \omega_{d,k} \wedge v \neq 0$. 
	Thus, by Lemma~\ref{lemma:injectivity} we have $L_{\ell}(v) = \omega_{d,\ell} \wedge v \neq 0$, which implies injectivity of $L_\ell$.
\end{proof}

One implication of Theorem~\ref{theorem:injectivity} and the property $\mathrm{LP}_{d,2}$ is the following.
\begin{corollary}
    Let $d \ge 3$ be odd and $\pmb{\lambda}$ be $d$-tuple of partitions of $m \le 2^{d-1}$ with parts at most $2$. Then for all $\ell > 2$ the map $L:  \mathrm{HWV}_{(\pmb{\lambda})'}\bigwedge V_{\ell} \to \mathrm{HWV}_{(\rho_\ell \pmb{\lambda})'}\bigwedge V_{\ell}$ is injective.
    In particular, $g(\pmb{\lambda}) \le g(\rho_\ell \pmb{\lambda})$.
\end{corollary}

\section{Positivity of rectangular Kronecker coefficients}
In this section we prove positivity of Kronecker coefficients $g_d(n,k)$ stated in Theorem~\ref{th:posafter}.

Recall that in particular, the established Lefschetz property for $k = 2$ implies that $\mathrm{AT}_d(2) \ne 0$. In \cite{ay22} we  showed that $\mathrm{AT}_d(2) > 0$ via signs of Latin hypercubes, and moreover, it can  be shown that every partial Latin hypercube in $[2]^d$ is positive.

The notions related to Latin hypercubes can be recalled in \textsection\ref{subsection:magicsets}.
We now define an operation on a magic sets to construct larger magic sets.
\begin{definition}
 The \textit{direct sum} $\oplus$ of magic sets $T_1$ and $T_2$ both of magnitude 
 $n$ in the cubes $[k_1]^d$ and $[k_2]^d$ respectively, is a magic set given by  
$$
	T_1 \oplus T_2 = T_1 \cup \{(a_1+k_1,\ldots,a_d+k_1): (a_1,\ldots,a_d) \in T_2\}
$$
of magnitude $n$ in the cube $[k_1+k_2]^d$. 
\end{definition}
A useful property of this operation is the  following.
\begin{proposition}
	For $T_1 \in B_{d,k_1}(n)$ and $T_2 \in B_{d,k_2}(n)$
	we have $$\mathrm{AT}(T_1\oplus T_2) = \mathrm{AT}(T_1)\cdot\mathrm{AT}(T_1).$$
\end{proposition}
\begin{proof}
Any partial Latin hypercube $L \in \mathcal{C}_d(T_1\oplus T_2)$ decomposes into a pair of Latin hypercubes $(L_1,L_2) \in \mathcal{C}_d(T_1) \times \mathcal{C}_d(T_2)$ with $\sgn(L) = \sgn(L_1)\,\sgn(L_2)$, since each slice of  $[k_1+k_2]^d$ intersects $T_1\oplus T_2$ either in $T_1$ or in $T_2$.
\end{proof}

\begin{theorem}
	For odd $d \ge 3$ and even $k$ we have $\omega^{2^{d-1}} \neq 0$.
\end{theorem}
\begin{proof} 
	Let $T = ([2]^d)^{\oplus \frac{k}{2}} \in B_{d,k}(n)$. Then from Theorem~\ref{th:omega} the coefficient of $\omega^{2^{d-1}}$ at $e_T$ up to a sign is equal to 
	$$\pm\mathrm{AT}(T) = \pm\mathrm{AT}_d(2)^{k/2} \neq 0$$ since $\mathrm{AT}([2]^d) = \mathrm{AT}_d(2) \ne 0$. 
\end{proof}
\begin{corollary}
	For odd $d \ge 3$ and even $k$ we have $g_d(n, k) > 0$ for all $n \le 2^{d-1}$.
\end{corollary}
\begin{proof}
	Since $\omega^{2^{d-1}} \neq 0$ then $\omega^{n} \neq 0$ for all $n \le 2^{d-1}$. But 
	$$g_d(n,k) = \dim \mathrm{HWV}_{(k\times n)^{d}} \bigwedge^{nk}(\mathbb{C}^{k})^{\otimes d}$$ and $\omega^n$ is an  element of latter space.
\end{proof}

\begin{remark}
	For fixed $d_0$ and sufficiently large $d$ by the above corollary we find that $(\omega_{d,k})^{k^{d_0-1}} \neq 0$ such that one of the coefficients up to a sign is equal to $\mathrm{AT}_{d_0}(k)$.
\end{remark}

\begin{remark}
	By analogous reasoning, for $n \le 2^{d-1}$ we can construct $\mathrm{SL}(n)^{\times d}$-invariant polynomial $\Delta_T \in \mathbb{C}[(\mathbb{C}^n)^{\otimes d}]_{2n}$ of minimal possible degree $2n$ (where $T \in B_{d,2}(n)$), such that $\Delta_T(I_n) = \mathrm{AT}(T) \neq 0$ (cf. \cite{ay22}),
	where $I_n = \sum_{i=1}^n (e_i)^{\otimes d}$ is the unit tensor (which is semistable for action of $\mathrm{SL}(n)^{\times d}$).
\end{remark}

\begin{corollary}
	For even $k$ and odd $d \ge 1 + \log_2 k$ we have $\omega^{k} \neq 0$.
\end{corollary}

\begin{remark}
From \cite{ay22} we know that one of the coefficients in expansion of $\omega^k$ is equal to $\mathrm{AT}_2(k)$.  Taking $d$ sufficiently large we also obtain that $\omega^k \neq 0$. One may consider relations between the coefficients in expansion of $\omega^k$ towards the Alon--Tarsi conjecture.
\end{remark}

\section{Some examples}

We illustrate some examples of sequences $\{g(\rho^n_k\, \pmb{\lambda})\}$.

\begin{example}\label{ex42}
Let us give some examples showing unimodality of $\{g(\rho^n_k\, \pmb{\lambda})\}$. 
\begin{itemize}
	\item[(a)] (Unimodal and symmetric) Let $\pmb{\lambda} = (42, 222, 321)$ so that for each partition $\lambda$ in this tuple we have $16\times 4 - {\lambda} = \rho^{13}_4 {\lambda}$, e.g. $16\times 4 - (42) = (4^{14}2)$. Then the sequence is unimodal and symmetric:
	$$
		\{g(\rho^n_4\, \pmb{\lambda})\}_{n=0,\ldots,13} = 
		1,15,128,728,2684,6395,9884,
		9884,6395,2684,728,128,15,1.
	$$
	\item[(b)] (Unimodal and not symmetric) Let $\pmb{\lambda} = (32, 221, 41)$. Then the sequence is unimodal but not symmetric: 
	$$
		\{g(\rho^n_4\, \pmb{\lambda})\}_{n=0,\ldots,13} = 
		1,8,54,281,1027,2531,4179,
		4584,
		3331,1613,521,114,18,2.
	$$
	\item[(c)] (Unimodality for odd $k$) While we discuss unimodality only for even $k$, it is frequently true for odd $k$ as well. For instance, let $\pmb{\lambda} = (32,221,311)$. Then the sequence is unimodal:
	$$\{g(\rho^n_3\, \pmb{\lambda})\}_{n=0,\ldots,6} = 1,4,7,7,5,3,1.$$ 
	Although for odd $k$ and $\pmb{\lambda} = (1 \times k)^{3}$ we know that $g(\pmb{\lambda}) = 1$ but $g(\rho_k\, \pmb{\lambda}) = 0$ which shows that the sequence is not always unimodal in this case.
\end{itemize}
\begin{example}[Log-concavity]
	Apart from unimodality, such sequences are frequently log-concave as well, that is when $g(\rho_k^n\, \pmb{\lambda})^2 \ge g(\rho_k^{n-1}\, \pmb{\lambda}) \cdot g(\rho_k^{n+1}\, \pmb{\lambda})$. For instance, the
	sequences from Example~\ref{ex42} are all log-concave. But this is not always true. For example, let $\pmb{\lambda} = ((1),(1),(1))$, then the following sequence is not log-concave but becomes so eventually:
	$$
	\{g(\rho^n_4\, \pmb{\lambda})\}_{n=0,\ldots,15} = 
	1,  1,  2, 6, 19,58,120,179,
	195,145,77,30,9, 2, 1, 1,
	$$
	where the log-concavity inequality holds for $n \in [4,12]$. It is interesting to find for which $\pmb{\lambda}$ the corresponding sequences of Kronecker coefficients are log-concave. 
	
	Another curious example is an `alternating' log-concavity observed in the following rectangular case:
	$$\{g_5(n, 2) = g(\rho^n_2\, \varnothing) \}_{n = 0, \ldots, 16} = 1, 1, 5, 11, 35, 52, 112, 130, 166, 130, 112, 52, 35, 11, 5, 1, 1,$$
	where the log-concavity inequality holds if $n$ is even and the reversed (log-convexity) inequality holds if $n$ is odd.
	Similarly, there is an example of an eventual `alternating' log-concavity in the following sequence:
	$$
		\{g_3(n,4) = g(\rho^n_4\, \varnothing)  \}_{n = 0,\ldots,16} = 
		1,1,1,2,5,6,13,14,18,14,13,6,5,2,1,1,1,
	$$
	where a similar alternating pattern holds for $n \in [4,12]$. 
	
	All sequences in our computations were observed to have such types of log-concavity properties. 
\end{example}
\end{example}

\section*{Acknowledgements}
The first author acknowledges the support by the Thematic Research Program ``Tensors: geometry, complexity and quantum entanglement", University of Warsaw, Excellence Initiative - Research University and Institute of Mathematics of the Polish Academy of Sciences, semester `Algebraic Geometry with Applications to Tensors and Secants'.


\begin{thebibliography}{abcdefghi}
\bibitem[AT92]{at}
N. Alon and M. Tarsi, Coloring and orientations of graphs, Combinatorica {\bf 12} (1992), 125--143.

\bibitem[AY22]{ay22}
A. Amanov and D. Yeliussizov, Fundamental invariants of tensors, Latin hypercubes, and rectangular Kronecker coefficients, Int. Math. Res. Not. {\bf 2023} (2023), 17552--17599.

\bibitem[BI13]{bi13}
P. B\"urgisser and C. Ikenmeyer, Explicit lower bounds via geometric complexity theory, In Proceedings of the forty-fifth Annual ACM symposium on Theory of Computing 2013, June. (pp. 141-150).

\bibitem[BI17]{bi}
P. B\"urgisser and C. Ikenmeyer, Fundamental invariants of orbit closures, J. Algebra {\bf 477} (2017), 390--434.

\bibitem[Bou05]{bourbaki}
N. Bourbaki, Lie groups and Lie algebras. Chapters 7–9. Springer-Verlag Berlin Heidelberg, 2005, 434~p. (orig. French ed.: \'El\'ements de Matematique, Groupes et Alg\`ebres de Lie 7-8 et 9, 1975 1982).

\bibitem[Cay43]{cay}
A. Cayley, On the theory of determinants, Trans. Cambridge Phil. Soc. VIII (1843), 1--16.

\bibitem[Cay45]{cay2}
A. Cayley, On the theory of linear transformations, Cambridge Math. J. {\bf 4} (1845), 193--209.

\bibitem[CHM09]{chm}
M. Christandl, A. W. Harrow, G. Mitchison, Nonzero Kronecker coefficients and what they tell us about spectra, Comm. Math. Phys. {\bf 270} (2007), 575--585.

\bibitem[Dri97]{dri1}
A. Drisko, On the number of even and odd Latin squares of order $p + 1$, Adv. Math. {\bf 128} (1997), 20--35.

\bibitem[Gly10]{glynn}
D. Glynn, The conjectures of Alon--Tarsi and Rota in dimension prime minus one, SIAM J. Discrete Math. {\bf 24} (2010), 394--399.


\bibitem[HR93]{hr}
R. Huang and G.-C. Rota, On the relations of various conjectures on Latin squares and straightening coefficients, Discrete Math. {\bf 128} (1994), 225--236.

\bibitem[Ike13]{iken13} 
C. Ikenmeyer, Geometric complexity theory, tensor rank, and Littlewood-Richardson coefficients, PhD diss., Universit\"{a}t Paderborn, 2013.

\bibitem[IMW17]{imw17} 
C. Ikenmeyer, K. D. Mulmuley, M. Walter, On vanishing of Kronecker coefficients, Comput. Complexity {\bf 26} (2017), 949--992. 

\bibitem[LZX21]{lzx}
X. Li, L. Zhang and H. Xia, Two classes of minimal generic fundamental invariants for tensors, arXiv preprint (2021), arXiv:2111.07343.

\end{thebibliography}
\end{document}